\documentclass[11pt]{amsart}
\usepackage{indentfirst,amssymb,amsmath,amsthm}
\newtheorem{theorem}{Theorem}[section]
\newtheorem{definition}{Definition}[section]
\newtheorem{lemma}{Lemma}[section]
\newtheorem{remark}{Remark}[section]
\newtheorem{corollary}{Corollary}[section]

\newtheorem{proposition}{Proposition}[section]

\newcommand{\be}{\begin{equation}}
\newcommand{\ee}{\end{equation}}
\newcommand{\bea}{\begin{eqnarray}}
\newcommand{\eea}{\end{eqnarray}}
\newcommand{\beas}{\begin{eqnarray*}}
\newcommand{\eeas}{\end{eqnarray*}}

\begin{document}
\setcounter{page}{1} \setlength{\unitlength}{1mm}\baselineskip
.58cm \pagenumbering{arabic} \numberwithin{equation}{section}

\title[]
{ Characterizations of a Lorentzian Manifold with a semi-symmetric metric connection}

\author[] 
{ Uday Chand De , Krishnendu De and S\.INEM G\"uler $^{*}$  }

\address
{Department of Pure Mathematics, University of Calcutta, West Bengal, India. ORCID iD: https://orcid.org/0000-0002-8990-4609}
\email {uc$_{-}$de@yahoo.com}

\address
 {Department of Mathematics,
 Kabi Sukanta Mahavidyalaya,
The University of Burdwan.
Bhadreswar, P.O.-Angus, Hooghly,
Pin 712221, West Bengal, India.\\
ORCID iD: https://orcid.org/0000-0001-6520-4520}
\email{krishnendu.de@outlook.in }

\address{Department of Industrial Engineering,
Istanbul Sabahattin Zaim University, Halkal\.i, Istanbul, Turkey.}
\email{sinem.guler@izu.edu.tr}

\begin{abstract}
In this article, we characterize a Lorentzian manifold $\mathcal{M}$ with a semi-symmetric metric connection. At first, we consider a semi-symmetric metric connection whose curvature tensor vanishes and establish that if the associated vector field is a unit time-like torse-forming vector field, then $\mathcal{M}$ becomes a perfect fluid spacetime. Moreover, we prove that if $\mathcal{M}$ admits a semi-symmetric metric connection whose Ricci tensor is symmetric and torsion tensor is recurrent, then $\mathcal{M}$ represents a generalized Robertson-Walker spacetime. Also, we show that if the associated vector field of a semi-symmetric metric connection whose curvature tensor vanishes is a $f-$ Ric vector field, then the manifold is Einstein and if the associated vector field is a torqued vector field, then the manifold becomes a perfect fluid spacetime. Finally, we apply this connection to investigate Ricci solitons.
\end{abstract}

\maketitle

\footnotetext{\par
\par
\thanks{$^{*}$ Corresponding author}

}

\section{Introduction}

In this study, we investigate several vector fields and solitons on a Lorentzian manifold $\mathcal{M}$ obeying a semi-symmetric metric connection (briefly, $\mathrm{SSMC}$). The notion of a semi-symmetric linear connection was first introduced by Friedman and Schouten on a differentiable manifold many years ago. The concept of metric connection with torsion tensor $T$ was then presented by Hayden \cite{hha} in 1932 on a Riemannian manifold. Yano \cite{ya1} made a methodical investigation of $\mathrm{SSMC}$ in 1970.\par

On a Lorentzian manifold $\mathcal{M}$, a linear connection $\nabla$ is named semi-symmetric if $T$ described by
\begin{equation*}
    T(U_{1},V_{1})=\nabla_{U_{1}}V_{1}-\nabla_{V_{1}}U_{1}-[U_{1},V_{1}]
\end{equation*}
 obeys
\begin{equation}\label{a1} T(U_{1},V_{1})=\pi (V_{1}) U_{1}-\pi (U_{1}) V_{1},\end{equation}
where $\pi$ is defined by $\pi(U_{1})=g(U_{1},\xi)$, for a fixed vector field $\xi$ (also, named associated vector field ). If we replace
$U_{1}$ and $V_{1}$, respectively, by $\phi U_{1}$ and $\phi V_{1}$ in the right side of (\ref{a1}), then the $\mathrm{SSMC}$ $\nabla$ changes into a quarter-symmetric connection \cite{go}, $\phi$ being a (1,1)- tensor field.\par

Again, if a $\mathrm{SSMC}$ $\nabla$ on $\mathcal{M}$ fulfills
\begin{equation}\label{a2}(\nabla_{U_{1}}g)(V_{1},Z_{1}) = 0,\end{equation}
then $\nabla$ is named metric \cite{ya1}, if not, it is non-metric \cite{hha}.\par

 We know \cite{ya1} that if $\nabla$ is a $\mathrm{SSMC}$ and $\mathrm{D}$ be a Levi-Civita connection, then
 \begin{equation}\label{a3}
\nabla_{U_{1}}V_{1} = \mathrm{D}_{U_{1}} V_{1} + \eta(V_{1})U_{1} - g(U_{1},V_{1})\xi.
\end{equation}
Further, it is also known \cite{ya1} that if $\overline{R}$ and $R$ denote the curvature tensors of $\nabla$ and $\mathrm{D}$, respectively, then
\begin{eqnarray}\label{a4}
  \overline{R}(U_{1},V_{1})Z_{1} &=& R(U_{1},V_{1})Z_{1} - L_1 (V_{1},Z_{1})U_{1} + L_1 (U_{1},Z_{1})V_{1}\nonumber\\&& - g(V_{1},Z_{1}) L_2 U_{1} + g(U_{1}, Z_{1})L_2 V_{1},
\end{eqnarray}
where $L_1$, a (0,2) tensor and $L_2$, a (1,1) tensor are defined by
\begin{equation}\label{a5}
L_1 (U_{1},V_{1})=(\mathrm{D}_{U_{1}}\pi)(V_{1}) - \pi(U_{1})\pi(V_{1}) +\frac{1}{2}\pi(\xi)g(U_{1},V_{1})
\end{equation}
and
\begin{equation}\label{a6}
 g(L_2 U_{1},V_{1}) = L_1 (U_{1},V_{1})
 \end{equation}
for any vector fields $U_{1}$ and $V_{1}$.\par
In a local coordinate system, equations (\ref{a4}) and (\ref{a5}) can be written
as follows:
\begin{equation}\label{a7}
\overline{R}_{hijk}= R_{hijk} -g_{hk}\pi _{ij} + g_{ik}\pi _{hj}- g_{ij}\pi _{hj} + g_{hk}\pi _{ik},
\end{equation}
\begin{equation}\label{a8}
\pi _{ik}=\mathrm{D}_{i}v_{k} - v_{i}v_{k} +\frac{1}{2}g_{ik}v,\;\;\;\;\;v=v^{h}v_{h}
\end{equation}
in which $\overline{R}_{hijk}$ and $R_{hijk}$ denote the curvature tensor of $\nabla$ and $\mathrm{D}$, respectively.\par
The idea of the $\mathrm{SSMC}$  was investigated by many researchers  (see, \cite{chaki}, \cite{ya1}, \cite{zen1}, \cite{zen2}, \cite{zh}, \cite{zh1}).\par

The current paper concerns with n-dimensional spacetimes (that is, a connected time-oriented Lorentz manifold) ($\mathcal{M},g$). For $\psi>0$, a smooth function (also called scale factor, or warping function), if $\mathcal{M}=-I \times_\psi,\mathrm{M}$, where $I$ is the open interval of $\mathbb{R}$ 
, $\mathrm{M}^{n-1}$ denotes the Riemannian manifold, then $\mathcal{M}$ is named a generalized Robertson-Walker $(\text{briefly}, GRW)$ spacetime \cite{alias1}. When the dimension of $\mathcal {M}$ is three and of constant sectional curvature, this spacetime represents a Robertson-Walker (briefly, $RW$) spacetime. In this context we may mention the work of  \"{U}nal and Shenawy on warped product manifolds (\cite{bus}, \cite{bus1}).\par

Due to the absence of a stress tensor and heat conduction terms corresponding to viscosity, the fluid is referred to as perfect and the energy momentum tensor $\mathcal{T}_{hk}$ is written by
\begin{equation}
\label{a9}
\mathcal{T}_{hk}=(\sigma+p)v_{h}v_{k}+p g_{hk},
\end{equation}
in which $p$ and $\sigma$ stand for the perfect fluid's isotropic pressure and energy density, respectively \cite{o'neill}.

For a gravitational constant $\kappa$, the Einstein's field equations without a cosmological constant is given by
\begin{equation}
\label{a10}
R_{hk}-\frac{R}{2}g_{hk}=\kappa \mathcal{T}_{hk},
\end{equation}
where $R_{hk}=R^{i}_{hki}$ and $R=g^{hk}R_{hk}$ are the Ricci tensor and the scalar curvature, respectively.\par

A spacetime $\mathcal{M}$ is named a perfect fluid (briefly $PF$) spacetime if the non-vanishing Ricci tensor $R_{hk}$ fulfills
\begin{equation}
\label{a11}
R_{hk}=\alpha g_{hk}+\beta v_{h} v_{k}
\end{equation}
where $\alpha$ and $\beta$ indicate smooth functions on $\mathcal{M}$. In the last equation, $g$ is the Lorentzian metric of $\mathcal{M}$ and $v_{h}$, the velocity vector is defined by $g_{hk}v^{h}v^{k}=-1$ and $v_{h}=g_{hk}v^{k}$. \par

Combining the equations (\ref{a9}), (\ref{a10}) and (\ref{a11}), we acquire
\begin{equation}
\label{a12}
\beta=k^2 (p+\sigma), \,\, \alpha=\frac{k^2 (p-\sigma)}{2-n}.
\end{equation}

In \cite{o'neill}, O'Neill established that a Robertson-Walker spacetime is a $PF$ spacetime. Every 4-dimensional generalized Robertson-Walker spacetime is a $PF$ spacetime if and only if it is a Robertson-Walker spacetime \cite{gtt}. For additional insights about the $PF$ and $GRW$ spacetimes, we refer( \cite{blaga2}, \cite{bychen}, \cite{bychen1}, \cite{de}, \cite{de5}, \cite{dug},  \cite{survey}) and its references.\par

\begin{definition}
For a non vanishing 1-form $\omega_{k}$ and a scalar function $\phi$ if the relation $\nabla_{k}u_{h} = \omega_{k}u_{h} + \phi g_{kh}$ holds, then the vector field $u$ is called torse-forming.
\end{definition}

This notion was introduced by Yano \cite{yano1} on a Riemannian manifold. It is noted that the foregoing torse-forming condition becomes $\nabla_{k}u_{h} = \phi(u_{k}u_{h}+ g_{kh})$, for a unit time-like vector.\par

In \cite{survey}, Mantica and Molinari and in \cite{bychen}, Chen have proven the subsequent theorems:\par
Let $\mathcal{M}$ be a Lorentzian manifold of dimension $n$ ($n \ge 3$).
\begin{theorem}
\label{2a}
\cite{survey} $\mathcal{M}$ is a $GRW$ spacetime if and only if it admits a unit time-like torse-forming vector field: $\nabla_{k}v_{j}=\varphi (g_{ij}+v_{j}v_{i})$, which is also an eigenvector of the Ricci tensor.
\end{theorem}

\begin{theorem}
\label{2b}
\cite{bychen} $\mathcal{M}$ is a $GRW$ spacetime if and only if it admits a time-like concircular vector field.
\end{theorem}

In this article, we consider a $\mathrm{SSMC}$ and prove the following:
\begin{theorem}
\label{3a}
If the associated vector field of $\mathcal{M}$ obeying a $\mathrm{SSMC}$ whose curvature tensor vanishes is a unit time-like torse-forming vector field, then $\mathcal{M}$ becomes a $PF$ spacetime.
\end{theorem}

\begin{definition}
In a semi-Riemannian manifold with a $\mathrm{SSMC}$ $\nabla$, the torsion tensor $T$ is called recurrent if it obeys $\nabla_{k}T^{h}_{ij}=A_{k}T^{h}_{ij}$, $A_{k}$ being a covariant vector.
\end{definition}
Here, we choose a $\mathrm{SSMC}$ whose torsion tensor is recurrent and establish the subsequent result:

\begin{theorem}
\label{4a}
Let $\mathcal{M}$ admit a $\mathrm{SSMC}$ whose Ricci tensor is symmetric and torsion tensor is recurrent. Then the manifold $\mathcal{M}$ represents a $\mathrm{GRW}$ spacetime.
\end{theorem}

\begin{definition}
 If $\nabla_{k}v_{h}=f R_{kh}$, then $v_{h}$ is named a $f$-Ric vector field, where $f$ is a constant.
\end{definition}

The foregoing idea of $f$-$Ric$ vector field was presented by Hinterleitner and Kiosak \cite{hinter} on a Riemannian manifold.

\begin{definition}
\cite{by}A vector field is called a torqued vector field if $\nabla_{k}v_{h}=f g_{kh}+ \omega_{k}v_{h}$ and $\omega_{k}v^{k}=0$, where $f$ is a real constant and 1-form $\omega$ is named the torqued form of $v$.
\end{definition}

In this article, we consider a $\mathrm{SSMC}$ whose associated vector field is $f-$ Ric and torqued vector field, respectively and prove the following results:
\begin{theorem}
\label{5a}
If $\mathcal{M}$ admits a $\mathrm{SSMC}$ whose curvature tensor vanishes and the associated vector field is $f-$ Ric, then $\mathcal{M}$ is Einstein.
\end{theorem}

\begin{theorem}
\label{6a}
Let $\mathcal{M}$ admit a $\mathrm{SSMC}$ whose curvature tensor vanishes. If the associated vector field is torqued, then $\mathcal{M}$ becomes a $PF$ spacetime.
\end{theorem}

Hamilton \cite{rsh2} introduces the concept of Ricci flow as a solution to the challenge of obtaining a canonical metric on a smooth manifold. Let 
$\mathcal{M}$ be satisfied by an evolution equation $\frac{\partial}{\partial t}g_{hk}(t)=-2R_{hk}$, then it is called Ricci flow \cite{rsh2}.  If the metric of $\mathcal{M}$ obeys the relation

\begin{equation}
\label{a13}
\mathfrak{L}_{v}g_{hk}+2R_{hk}+2\overline{\lambda} g_{hk}=0,
\end{equation}

then it is called a Ricci soliton \cite{rsh1}, where $\mathfrak{L}_{v}$ is the Lie derivative operator and $\overline{\lambda}$ denotes a real constant. Here, $v$ is called the potential vector field of the solitons. The Ricci soliton is named shrinking, expanding or steady, if $\overline{\lambda}$ is negative, positive, or zero, respectively.\par

Many mathematicians were fascinated in the geometry of Ricci solitons over the past few years. Some interesting findings on this solitons have been investigated in (\cite{blaga2}, \cite{chde}, \cite{deshmukh2}, \cite{rsr5}) and also by others.\par

The above investigations motivate us to $\mathrm{SSMC}$ and establish the subsequent:
\begin{theorem}
\label{7a}
Let us suppose that $\mathcal{M}$ admit a $\mathrm{SSMC}$ whose curvature tensor vanishes and the associated vector field is a unit time-like torse-forming vector field. If $\mathcal{M}$ admits a Ricci soliton $(g, v, \overline{\lambda})$, then the soliton is shrinking if $\phi<1$, expanding if $\phi>1$, or steady if $\phi = 1$, respectively. Additionally, the integral curves that $v$ generates are geodesics.
\end{theorem}

\section{Semi-symmetric metric connection}
$\mathbf{Agreement}$: Throughout this article, we assume that the associated vector $v_{i}$ is a unit time-like vector, that is, $v^{i}v_{i}=-1$.\par

Then, the equation (\ref{a8}) becomes
\begin{equation}\label{b1}
\pi_{ik}=\mathrm{D}_{i}v_{k} - v_{i}v_{k} -\frac{1}{2}g_{ik}.
\end{equation}
Multiplying (\ref{a7}) with $g^{hk}$, we acquire
\begin{equation}\label{b2}
\overline{R}_{ij}=R_{ij} -(n-2)\pi _{ij} -\pi g_{ij},
\end{equation}
where $\overline{R}_{ij}$ and $R_{ij}$ are the Ricci tensor of $\nabla$ and $\mathrm{D}$, respectively and $\pi= g^{hk}\pi_{hk}$.\par
Also, we have
\begin{equation}\label{b3}
\nabla _{h} v_{k}=\mathrm{D}_{h}v_{k} - v_{h}v_{k} - g_{hk}.
\end{equation}
\begin{proposition}
$\overline{R}_{ij}$ is symmetric if and only if $v_{i}$ is closed.
\end{proposition}
\begin{proof}
From the equation (\ref{b2}), we obtain
\begin{eqnarray}\label{b4}
\overline{R}_{ij}-\overline{R}_{ji}&=&-(n-2)\{ \pi _{ij}-\pi_{ji}\}\nonumber\\&&
=-(n-2)\{\mathrm{D}_{i}v_{j} - v_{i}v_{j} -\frac{1}{2}g_{ij}-\mathrm{D}_{j}v_{i} + v_{j}v_{i} +\frac{1}{2}g_{ji} \}
\nonumber\\&&=-(n-2)\{\mathrm{D}_{i}v_{j}-\mathrm{D}_{j}v_{i}\}.
\end{eqnarray}
If $v_{i}$ is closed, then the foregoing equation reduces to
\begin{equation}\label{b5}
\overline{R}_{ij}-\overline{R}_{ji}=0,
\end{equation}
that is, $\overline{R}_{ij}$ is symmetric and conversely.
\end{proof}

\subsection{Proof of the Theorem~{\upshape\ref{3a}}}

Since by hypothesis, $\overline{R}_{hijk}=0$, from (\ref{a7}), we get
\begin{equation}\label{b7}
R_{hijk} =g_{hk}\pi _{ij} - g_{ik}\pi _{hj}+ g_{ij}\pi _{hj} - g_{hk}\pi _{ik}.
\end{equation}
Let
\begin{equation}\label{b8}
\nabla_{k}v_{j}=\phi (g_{kj}+v_{j}v_{k}).
\end{equation}
Then from (\ref{b3}) it follows that
\begin{equation}\label{b9}
\mathrm{D}_{k}v_{j}= (\phi +1) (g_{kj}+v_{j}v_{k}).
\end{equation}
Hence, using the foregoing equation in (\ref{b1}), we obtain
\begin{equation}\label{b10}
\pi_{ik}=(\phi -\frac{1}{2}) g_{ik}+\phi v_{i}v_{k}).
\end{equation}
Now, making use of the previous equation in (\ref{b7}), we acquire
\begin{equation}\label{b11}
R_{hijk} =(2\phi-1)\{g_{hk}g _{ij} - g_{hj}g_{ik}\}+\phi \{g_{hk}v_{i}v_{j}+g_{ij}v_{h}v_{k}-g_{ik}v_{h}v_{j}-g_{hj}v_{i}v_{k}\}.
\end{equation}
Multiplying with $g^{ij}$ the above equation yields
\begin{equation}\nonumber
R_{hk} =(2\phi-1)(n-1)g_{hk}+\phi \{-g_{hk}+n v_{h}v_{k}-v_{h}v_{k}-v_{h}v_{k}\},
\end{equation}
which implies
\begin{equation}\label{b12}
R_{hk} =\{(n-1)(2\phi-1)-\phi \}g_{hk}+(n-2) \phi v_{h}v_{k}.
\end{equation}
The foregoing equation entails that 
 $\mathcal{M}$ is a $PF$ spacetime.\par

Hence, the proof is finished.\par

Now, using Theorem \ref{2a}, we may state:

\begin{corollary}
\label{cor 1}
A $\mathrm{GRW}$-spacetime admitting a $\mathrm{SSMC}$ whose curvature tensor vanishes is a $PF$ spacetime.
\end{corollary}
\begin{remark}
In \cite{survey}, Mantica and Molinari established that a $GRW$ spacetime satisfying $div C=0$ represents a $PF$ spacetime, where $C$ indicates the conformal curvature tensor.
Here, curvature condition $div C=0$ is replaced by a $\mathrm{SSMC}$ whose curvature tensor vanishes and acquires the same result.
\end{remark}

For $n=4$, comparing the equations (\ref{a11}) and (\ref{b12}), we have
\begin{equation}\label{b13}
 \alpha g_{hk}+ \beta v_{h}v_{k}=(5\phi-3)g_{hk}+2 \phi v_{h}v_{k}.
\end{equation}
Making use of (\ref{a12}), the foregoing equation yields

\begin{equation}\label{b14}
    k^2 (3p-2\sigma)=3,
\end{equation}

which implies

\begin{equation}\label{b15}
    (3p-2\sigma)=c\,\, (say).
\end{equation}

\begin{corollary}
\label{cor 2}
A $PF$ spacetime admitting a $\mathrm{SSMC}$ satisfies the state equation $\sigma = \frac{3}{2}p +$ constant.
\end{corollary}

\subsection{Proof of the Theorem~{\upshape\ref{4a}}}
It is known \cite{chaki} that if a Riemannian manifold admits a $\mathrm{SSMC}$ $\nabla$ whose torsion tensor $T$ is recurrent, then the curvature tensor $\overline{R}^{h}_{ijk}$ of $\nabla$ is given by

\begin{eqnarray}\label{c1}
\overline{R}^{h}_{ijk}&=& R^{h}_{ijk} +A_{k}[v_{i}\delta ^{h}_{j}-g_{ij}v^{h}]\nonumber\\&&
 - A_{j}[v_{k}\delta ^{h}_{i}-g_{ik}v^{h}]+v^{l}v_{l}[g_{jk}\delta ^{h}_{i}-g_{ik}\delta ^{h}_{j}].
\end{eqnarray}
For a Lorentzian manifold the above relation becomes
\begin{eqnarray}\label{c2}
\overline{R}^{h}_{ijk}&=& R^{h}_{ijk} +A_{k}[v_{i}\delta ^{h}_{j}-g_{ij}v^{h}]\nonumber\\&&
 - A_{j}[v_{k}\delta ^{h}_{i}-g_{ik}v^{h}]-[g_{jk}\delta ^{h}_{i}-g_{ik}\delta ^{h}_{j}],
\end{eqnarray}
since $v^{l}v_{l}=-1$.\par
First we prove the following lemma:
\begin{lemma}\label{l1}
If the torsion tensor $T$ is given by
\begin{equation}\label{lk1}
 \nabla_{k}T^{h}_{ij}=A_{k}T^{h}_{ij},
\end{equation}
then $A_{k}$ is gradient.
\end{lemma}
\begin{proof}
Let \begin{equation}\label{lk2}
f^{2}=T_{hij}T^{hij},
\end{equation}
where $T_{hij}=g_{hk}T^{k}_{ij}$.

From (\ref{lk2}), we get
\begin{eqnarray}\label{lk4}
2f \nabla_{l}f&=&\nabla_{l}T_{hij}T^{hij}+T_{hij}\nabla_{l}T^{hij}\nonumber\\&&
=2A_{l}f^{2},
\end{eqnarray}
which implies that
\begin{equation}\label{lk5}
\nabla_{l}f=fA_{l}.
\end{equation}
Therefore, from the above equation, we obtain
\begin{equation}\label{lk6}
\nabla_{l}\nabla_{m}f-\nabla_{m}\nabla_{l}f=\nabla_{m}fA_{l}-\nabla_{l}fA_{m}+f(\nabla{m}A_{l}-\nabla_{l}A_{m}).
\end{equation}
Since $f$ is a scalar, using the equation (\ref{lk5}) in (\ref{lk6}), we acquire
\begin{equation}\label{lk7}
f(\nabla{m}A_{l}-\nabla_{l}A_{m})=0.
\end{equation}
Since $f\neq 0$, $A_{l}$ is gradient.
\end{proof}
From $T^{h}_{ij}=\delta ^{h}_{i}v_{j}-\delta ^{h}_{j}v_{i}$, we get
\begin{equation}\label{c3}
T^{h}_{ij}=(n-1)v_{j},
\end{equation}
which implies
\begin{equation}\label{c4}
\nabla_{k}T^{h}_{ij}=(n-1)v_{k}v_{j}.
\end{equation}
Again from $\nabla_{k}T^{h}_{ij}=A_{k}T^{h}_{ij}$, we obtain
\begin{equation}\label{c5}
\nabla_{k}T^{h}_{ij}=A_{k}T^{h}_{ij}=(n-1)A_{k}v_{j}.
\end{equation}
Contracting $h$ and $k$ in (\ref{c2}) yields
\begin{equation}\label{c6}
\overline{R}_{ij}= R_{ij} +[A_{j}v_{i}-g_{ij}A_{h}v^{h}].
\end{equation}
By hypothesis $\overline{R}_{ij}$ is symmetric. Hence, we have from (\ref{c6})
\begin{equation}\label{c7}
A_{i}v_{j}=A_{j}v_{i}.
\end{equation}
Multiplying the foregoing equation by $v^{i}$ gives

\begin{equation}\label{c8}
A_{j}=a v_{j},
\end{equation}

where $a=-A_{i}v^{i}$.\par
From (\ref{c4}) and (\ref{c5}) it follows that

\begin{equation}\label{c9}
\nabla_{k} v_{j}=A_{k}v_{j}.
\end{equation}

Using (\ref{b3}) in the previous equation yields

\begin{equation}\label{c10}
\mathrm{D}_{h}v_{k} =g_{hk}+\mu _{h}v_{k},
\end{equation}

where $\mu_{h}=A_{h}+v_{h}$.\par
From the above equation we conclude that $v_{k}$ is a torseforming vector field with $\alpha=\beta=1$.\par
Here, $v_{h}$ is closed, since by hypothesis the Ricci tensor of the $\mathrm{SSMC}$ is symmetric. Also by Lemma \ref{l1}, $A_{h}$ is closed and hence $\mu_{h}$ is closed. Therefore, $v_{h}$ is concircular. Hence, using Theorem \ref{2b} we state that 
$\mathcal{M}$ represents a $GRW$ spacetime.\par
This completes the proof.

\subsection{Proof of the Theorem~{\upshape\ref{5a}}}

Since the associated vector field is $f-$ Ric, then we have
\begin{equation}\label{y1}
\nabla_{k}v_{h}=fR_{hk},
\end{equation}
$f$ is a constant.\par
Then from (\ref{b3}), we acquire
\begin{equation}\label{y2}
\mathrm{D}_{h}v_{k}= fR_{hk}+ g_{hk}+v_{h}v_{k}.
\end{equation}
Hence, using the previous equation in (\ref{b1}), we infer
\begin{equation}\label{y3}
\pi_{ik}=fR_{ik}+\frac{1}{2} g_{ik}.
\end{equation}

Since by assumption $\overline{R}_{hijk}=0$. From (\ref{a7}), we obtain
\begin{equation}\label{y4}
R_{hijk} =g_{hk}\pi _{ij} - g_{ik}\pi _{hj}+ g_{ij}\pi _{hj} - g_{hk}\pi _{ik}.
\end{equation}

Now, using  equation (\ref{y3}) in (\ref{y4}), we get
\begin{equation}\label{y5}
R_{hijk} =f\{ g_{hk}R_{ij}-g_{ik}R_{hj}+g_{ij}R_{hk}-g_{hj}R_{ik}\}+\{g_{hk}g _{ij} - g_{hj}g_{ik}\}.
\end{equation}
Multiplying with $g^{ij}$ the foregoing equation yields

\begin{equation}\nonumber
R_{hk} =f\{ n R_{hk}-R_{hk}+Rg_{hk}-R_{hk}\}+(n-1)g _{hk},
\end{equation}
which entails
\begin{equation}\label{y6}
R_{hk} =\frac{f R+(n-1)}{1-(n-2)f}g_{hk}.
\end{equation}
The above equation states that 
$\mathcal{M}$ is Einstein.\par
This ends the proof.\par
\begin{remark}
From the equation (\ref{y5}), we also get the spacetime is of constant curvature $\frac{2f^{2} R+nf+1}{1-(n-2)f}$.
\end{remark}
If $R=-\frac{nf+1}{f^{2}}$, then...

\subsection{Proof of the Theorem~{\upshape\ref{6a}}}
By hypothesis, 
$v$ is a torqued vector field. Therefore, we get
\begin{equation}\label{z1}
\nabla_{k}v_{h}=fg_{hk}+\omega_{k}v_{j}
\end{equation}
and $\omega_{k}v^{k}=0$.
Then from (\ref{b3}), we obtain
\begin{equation}\label{z2}
\mathrm{D}_{h}v_{k}= (f+1)g_{hk}+(\omega_{h}+v_{h})v_{k}.
\end{equation}
Hence, using the above equation in (\ref{b1}), we acquire
\begin{equation}\label{z3}
\pi_{ik}=(f+\frac{3}{2}) g_{ik}+\omega_{i}v_{k}.
\end{equation}

Since $\overline{R}_{hijk}=0$, from (\ref{a7}), we get
\begin{equation}\label{z4}
R_{hijk} =g_{hk}\pi _{ij} - g_{ik}\pi _{hj}+ g_{ij}\pi _{hj} - g_{hk}\pi _{ik}.
\end{equation}

Now, using  equation (\ref{z3}) in (\ref{z4}), we have
\begin{equation}\label{z5}
R_{hijk} =2(f+\frac{3}{2})\{g_{hk}g _{ij} - g_{hj}g_{ik}\} +\{ g_{hk}\omega_{i}-g_{ik}\omega_{h}\}v_{j}+\{g_{ij}\omega_{h}-g_{hj}\omega_{i}\}v_{k}.
\end{equation}
Multiplying with $g^{ij}$, we acquire

\begin{equation}\nonumber
R_{hk} =2n(f+\frac{3}{2})g_{hk}- 2(f+\frac{3}{2})g_{hk} -(n-1)\omega_{h}\omega_{k}+\{g_{hk}\omega_{h}v_{k}-\omega_{h}v_{k}\}.
\end{equation}

\begin{equation}\label{z6}
R_{hk} =(2n-2)(f+\frac{3}{2})g_{hk}+(n-2)\omega_{h}v_{k}.
\end{equation}
Again, multiplying with $v^{h}$ yields
\begin{equation}\label{z7}
R_{hk}v^{h} =(2n-2)(f+\frac{3}{2})v_{k},
\end{equation}
since $\omega_{i}v^{i}=0$.\par
Also, from (\ref{z6}), we get $\omega_{h}v_{k}=\omega_{k}v_{h}$. Multiplying the foregoing relation by $v^{k}$, we have $\omega_{h}=av_{h}$, where $a=\omega_{k}v^{k}$.
Hence, the equation (\ref{z6}) becomes
\begin{equation}\label{z8}
R_{hk} =(2n-2)(f+\frac{3}{2})g_{hk}+(n-2)a v_{h}v_{k}.
\end{equation}
Foregoing equation tells that 
$\mathcal{M}$ is a $PF$ spacetime.\par

This accomplishes the proof.\par

\section{Applications to Ricci solitons }

\subsection{Proof of the Theorem~{\upshape\ref{7a}}}

Now, suppose that  $\mathcal{M}$ with $\mathrm{SSMC}$ admits a Ricci soliton. Then from equation (\ref{a13}), we get

\begin{equation}
\label{r1}
\nabla_{h}v_{k}+\nabla_{k}v_{h}+2R_{hk}+2\overline{\lambda} g_{hk}=0.
\end{equation}
Multiplying by $v^{h}v^{k}$, we obtain
\begin{equation}
\label{r2}
R_{hk}v^{h}v^{k}-\overline{\lambda} =0,
\end{equation}
since $g_{hk}v^{h}v^{k}=-1$ and $v^{h}v^{k}\nabla_{h}v_{k}=0$. Again, multiplying the equation (\ref{b12}) by $v^{h}v^{k}$,  we infer
\begin{equation}
\label{r3}
R_{hk}v^{h}v^{k}=-\{(n-1)(2\phi-1)-\phi \} + (n-2) \phi.
\end{equation}

Comparing the last two equations, we acquire $\overline{\lambda}=(1-n)(\phi-1)$. This shows that the Ricci soliton is shrinking if $\phi<1$, expanding if $\phi>1$, or steady if $\phi = 1$, respectively.\par

 Again, multiplying (\ref{r1}) by $v^{h}$, we find
\begin{equation}
\label{r4}
v^{h}\nabla_{k}v_{h}+2R_{hk}v^{h}+2\overline{\lambda} v_{k}=0,
\end{equation}
since $v^{h}\nabla_{h}v_{k}=0$.\par
Again, multiplying (\ref{b12}) by $v^{h}$ yields
\begin{equation}
\label{r5}
R_{hk}v^{h}=\{(n-1)(2\phi-1)-\phi \} - (n-2) \phi.
\end{equation}
Using the above equation in (\ref{r4}), we have $v^{h}\nabla_{k}v_{h}=0$, since $\overline{\lambda}=(1-n)(\phi-1)$.\par
This ends the proof.

\end{document}